\documentclass[11pt]{amsart}
\usepackage{amsmath,amssymb,latexsym, dsfont}
\usepackage[small]{caption}
\usepackage{graphicx,color,mathrsfs,tikz}
\usepackage{subfigure,color,bm}
\usepackage[colorlinks=true,urlcolor=blue,
citecolor=red,linkcolor=blue,linktocpage,pdfpagelabels,
bookmarksnumbered,bookmarksopen]{hyperref}
\usepackage[italian,english]{babel}
\usepackage{units}
\usepackage{enumitem}
\usepackage[left=2.5cm,right=2.5cm,top=2.9cm,bottom=2.9cm]{geometry}
\usepackage{fancyhdr, dsfont}

\usepackage{amscd}
\usepackage{esint}

\theoremstyle{plain}

\newtheorem{thm}{Theorem}[section]
\newtheorem*{thm*}{Theorem}
\theoremstyle{plain}
\newtheorem{lem}[thm]{Lemma}
\newtheorem{pro}{Proposition}[section]

\theoremstyle{definition}

\newtheorem{rem}{Remark}[section]

\newcommand{\tu}{\tilde u}

\newcommand{\I}{J}

\newcommand{\J}{K}

\newcommand{\mN}{\mathbb{N}}
\newcommand{\mS}{\mathbb{S}}
\newcommand{\mR}{\mathbb{R}}

\newcommand{\ro}{\mathbb{R}}

\newcommand{\mcal}[1]{\mathcal{#1}}
\newcommand{\md}[1]{\left|#1 \right|}

\newcommand{\bct}[1]{\left(#1\right)}

\newcommand{\T}[1]{\tilde{#1}}

\newcommand{\dsp}{\displaystyle}

\numberwithin{equation}{section} \allowdisplaybreaks

        \title[Exponential Integrability in the spirit of Moser-Trudinger's inequalities ]{Exponential integrability in the spirit of Moser-Trudinger's inequalities of functions with finite non-local, non-convex energy}

        \author[A. Mallick]{Arka Mallick}
        \address[A. Mallick]{Department of Mathematics \newline\indent
	EPFL SB CAMA \newline\indent
	Station 8 CH-1015 Lausanne, Switzerland}
\email{\tt arka.mallick@epfl.ch}

        \author[H.-M. Nguyen]{Hoai-Minh Nguyen}
  \address[H.-M. Nguyen]{Department of Mathematics \newline\indent
	EPFL SB CAMA \newline\indent
	Station 8 CH-1015 Lausanne, Switzerland}
\email{\tt hoai-minh.nguyen@epfl.ch}
        
\keywords{Sobolev's inequality, Poincar\'e's inequality, Moser-Trudinger's inequality.}
\subjclass[2010]{26D10, 26A54}

\begin{document}

\begin{abstract}  Let $d \ge 1$, $p \ge d$, and let $\Omega$ be a smooth bounded open subset of $\mR^d$. 
We prove some exponential integrability   in the spirit of Moser-Trudinger's inequalities for measurable functions $u$ defined in $\Omega$ such that 
$$
\mathop{\int_{\Omega} \int_{\Omega}}_{|u(x) - u(y)| > \delta} \frac{1}{|x-y|^{d+p}} \, dx \, dy < + \infty, 
$$
for some $\delta > 0$. This double integral appeared in  characterizations of Sobolev spaces and involved in improvements of the Sobolev inequaliies, Poincar\'e inequalities, and Hardy inequalities. 
\end{abstract}

\maketitle


\section{Introduction}

Let $(\rho_n)$ be  a sequence of non-negative radial functions satisfying
\begin{equation}\label{rho-n}
\dsp \lim_{ n \to \infty } \int_\tau^\infty \rho_n(r) r^{N-1} \,
dr = 0 \quad   \forall \, \tau >0,
\quad \mbox{ and } \quad 
\lim_{n \to  \infty} \int_{0}^{+ \infty} \rho_n(r) r^{N-1} \, dr =
1.
\end{equation}
Set 
\begin{equation}\label{def-Kdp}
K_{d, p}: = \int_{\mS^{d-1}} |\sigma \cdot e|^p \, d \sigma,  
\end{equation}
for some $e \in \mS^{d-1}$, the unit sphere in $\mR^d$. 

\medskip 
Jean Bourgain, Haim Brezis, and Petru Mironescu \cite[Theorems 1 and 2]{BBM1} (see also \cite{BBM2} and \cite{B1}) proved the following BBM formula: 

\begin{pro} \label{proBBM}
Let $d \ge 1$, $p>1$ and  let $\Omega$ be a smooth bounded open subset in $\mR^d$ or $\Omega = \mR^d$. Assume that  $g \in L^p(\Omega)$ and let $(\rho_n)$ satisfy \eqref{rho-n}. Then $g \in W^{1, p}(\mR^d)$ if and only if 
\begin{equation*}
\liminf_{n  \to  \infty } \int_{\Omega} \int_{\Omega} \frac{|u(x)
-u(y)|^p}{|x-y|^p} \rho_n(|x-y|)\, dx \, dy  < + \infty.
\end{equation*}
Moreover, for $g \in W^{1, p}(\Omega)$, 
\begin{equation*}
\lim_{n  \to  \infty } \int_{\Omega} \int_{\Omega} \frac{|u(x)
-u(y)|^p}{|x-y|^p} \rho_n(|x-y|)\, dx \, dy = K_{d, p}
\int_{\mR^N} |\nabla u(x)|^p \, dx.  
\end{equation*}
\end{pro}

A variant of Propposition~\ref{proBBM} for $p =1$ involving functions of bounded variations was obtained by  Jean Bourgain, Haim Brezis, and Petru Mironescu \cite{BBM1} and Juan Davila \cite{Davila}. Further studies related to this characterizations can be founded in \cite{ADM, AK, BN-Two, BN-BBM, CV, LeoniSpector, Ng-Squ1, Ng-Squ4, Squ1, Po1, Po2, PoSp}.

We next discuss another characterization of Sobolev spaces in the spirit of the BBM formula. To this end,  for $d\ge 1$, $p \ge 1$,  and $\delta > 0$, for a measurable subset $O$ of $\mR^d$, and for a measurable function $u$ defined in $O$, set
\begin{equation}
I_{\delta, p} (u, O) = \mathop{\mathop{\int_O \int_O}}_{|u(x) - u(y)|> \delta} \frac{ \delta^p}{|x - y|^{d + p}} \, dx \, dy.  
\end{equation}
This quantity has its root in  estimates for the topological degree in \cite{BBNg, BBM3, NgDegree1, NgDegree2, Sch} which has the motivation from the study of the Ginzburg Landau equation \cite{BBH}.

\medskip 
It was shown \cite[Theorems 2 and 5]{NgSob1} and \cite[Theorem 1]{BoNg1} that 
\begin{pro}\label{proPrincipal}
Let $d \ge 1$ and $\Omega$ be a smooth bounded open subset in $\mR^d$ or $\Omega = \mR^d$ and  let  $p> 1$ and  $g \in L^p(\Omega)$. Then $u \in W^{1, p}(\Omega)$ if and only if 
\begin{equation*}
\liminf_{\delta \to 0 } I_{\delta, p} (u, \Omega) < + \infty. 
\end{equation*}
Moreover,  for $g \in W^{1, p}(\Omega)$, 
\begin{equation}\label{equality1}
\lim_{\delta \to 0 } I_{\delta, p} (u, \Omega) =   \frac{1}{p} K_{d, p} \int_{\Omega} |\nabla u(x)|^p
\,dx,  
\end{equation}
where $K_{d,p}$ is defined by \eqref{def-Kdp}. We also have, for all $\delta > 0$, 
\begin{equation} \label{estimate1.1}
I_{\delta, p}(u, \Omega) \le C_{d,p} \int_{\mR^N} |\nabla u(x)|^p \, dx \quad \forall \,  u \in W^{1, p}(\Omega), 
\end{equation}
for some positive constant $C_{d, p}$ depending only on $d$ and $p$.  
\end{pro}

The case $p=1$ is more delicate. One has \cite[Theorem 8]{NgSob1} (see also \cite[Proposition 2]{BreHmn}),  for $u \in W^{1, 1}(\Omega)$, 
\begin{equation*}
\liminf_{\delta \to 0} I_{\delta, 1} (u, \Omega) \ge K_{d,1} \int_{\Omega} |\nabla u| \, dx  
\end{equation*}
and  (see \cite[Theorem 8]{NgSob1} and \cite[Theorem 1]{BoNg1})  that  $u \in BV(\Omega)$ provided that  $u \in L^1(\Omega)$ and  $\liminf_{\delta \to 0} I_{\delta, 1}(u, \Omega) <  + \infty$. 
Let $B_r$ denote the ball centered at 0 and of radius $r$. An example due to Augusto Ponce presented in \cite{NgSob1} showed that there exists $u \in W^{1, 1}(B_1)$  such that $\lim_{\delta \to 0} I_{\delta, 1}(u, B_1) = + \infty$.  When $d=1$,  there exists $u \in W^{1,1}(0, 1)$ \cite[Pathology 2]{BreHmn}  such that 
\begin{equation*}
K_{1,1} \int_{0}^1 |\nabla u| \, dx  = \liminf_{\delta \to 0} I_{\delta, 1} (u, (0, 1)) < \limsup_{\delta \to 0} I_{\delta, 1}(u, (0, 1)) = + \infty. 
\end{equation*}
It turns out that the concept of $\Gamma$-convergence  fits very well this setting.  It was shown \cite{NgGammaCRAS, NgGamma} that the $\Gamma$-limit exists for $p \ge 1$.  Surprisingly, the $\Gamma$-limit,  which is positive,  is strictly less than the pointwise limit \cite{NgGamma, NgGammaCRAS}. The quantity  $I_{\delta, 1}$ has a similar form with  non-local filters using in denoising process \cite{BuadesCollMorel2},  in particular with  Yaroslavsky's ones \cite{Yaroslavsky1, YaroslavskyEden}.  A discussion on a connection between nonlocal filters using $I_{\delta, 1}$ and local ones involving the total variations via the $\Gamma$-convergence theory 
 is given in \cite[Section 5.2]{BreHmn}.   Further interesting investigations related to the $\Gamma$-limit of $I_{\delta, p}$ are  given in \cite{AGMP1, AGMP2, AGP,BreHmn}.  

\medskip
One can obtain new and improved  variants of Poincar\'e's inequality, Sobolev's inequality and Rellich-Kondarachov's compactness criterion using the information of $I_{\delta, p}$ instead of the one of  the gradient  \cite[Theorems 1, 2, and 3]{NgSob3}.  Concerning the Sobolev inequality, one has 

\begin{pro} Let $1 < p < d$ and set $q  = d p / (d - p)$ and fix $\delta > 0$ arbitrary. We have, for $u \in L^p(\mR^d)$,  
\begin{equation}\label{pro-Sobolev} 
\Big(\int_{|u| > \lambda \delta} |u|^q \Big)^{1/q} \le C \Big( I_{\delta, p}  (u, \mR^d) \Big)^{1/p}, 
\end{equation}
for some positive constants $\lambda$ and $C$ independent of $u$. 
\end{pro}

Concerning the Poincar\'e inequality, one obtains 

\begin{pro}
Let $d\geq 1$, $p\geq 1$, $\delta> 0$,  let $B$ be an open ball of $\mR^d$, and let $u \in L^p(B)$.  There exists a positive constant $C_{d,p}$ depending only on $d$ and $p$ such that 
\begin{equation} \label{pro-Poincare}
\int_B\int_B |u(x)-u(y)|^pdxdy \leq C_{d,p} \bct{|B|^{\frac{d+p}{d}}I_{\delta,p}(u,B)+ \delta^p|B|^2}.
\end{equation}
\end{pro}
The proof of Sobolev's inequality \eqref{pro-Sobolev} is based on the one of Poincar\'e's inequality \eqref{pro-Poincare} and uses the theory of sharp functions due to Charles Fefferman and Elias Stein \cite{FS} and the method of truncation due to Vladimir Mazya \cite{Ma}. The proof of Poincar\'e's inequality \eqref{pro-Poincare} has its roots in \cite{BoNg1} and uses John-Nirenberg's inequality \cite{JN}. 

\begin{rem} For a measurable function defined in $B$, by applying \eqref{pro-Poincare} for $u_k$ with \\$u_k = \min\big\{k, \max\{u, -k\} \big\}$ and letting $k\to + \infty$, one also obtains \eqref{pro-Poincare} for measurable functions.  
\end{rem}

With Marco Squassina, the second author also established  new and improved variants of Hardy and Caffarelli, Kohn, Nirenberg's inequality \cite{Ng-Squ1} using the quantity $I_{\delta, p}$. The approach used in \cite{Ng-Squ1} does not involve  the integration-by-parts arguments and can be extended for the fractional Sobolev spaces  \cite{Ng-Squ2}.  Other investigations related to $I_{\delta, p}$ can be found in \cite{BrezisNguyen, BreHmn, NgSob2, NgSob3,Ng-Squ2, Ng-Squ3, Ng-Squ4}.

Let   $\Omega$ be a smooth bounded open subset of $\mR^d$ and  $p \ge  d$. It follows from \eqref{pro-Poincare} that $u \in BMO(\Omega)$ provided that  $u \in L^1(\Omega)$ and $I_{\delta, p}(u, \Omega) < + \infty$. More precisely, one has, for $p \ge d$,  
\begin{equation*}
\| u \|_{BMO(\Omega)} \le C_{\Omega} \Big( I_{\delta, p}(u, \Omega) + \delta^d \Big), 
\end{equation*}
where
\begin{equation}\label{def-BMO}
\| u\|_{BMO(\Omega)} : = \sup_{\mathrm{ ball } \; B \subset \Omega} \frac{1}{| B|} \int_{B} |u - u_{B}| \, dx. 
\end{equation}
Here, for a given a measurable set $O$ of $\mR^d$ and a function $u \in L^1(O)$, one sets
\begin{equation}
|O| : = \mbox{meas} (O) \quad \mbox{ and } \quad u_O =\fint_{O} u \, dx \mbox{ with } \fint_{O} u \, dx : =  \frac{1}{|O|} \int_O u \, dx.
\end{equation}
One can then derive the exponential integrability of $u$ from John-Nirenberg's  inequality: 
\begin{equation}\label{John-Nirenberg}
\fint_{B} e^{c |u - u_B|/ \| u \|_{BMO(B)} } \le C,  
\end{equation}
for some positive constant $c$ and $C$ depending only on $d$ and for any open ball $B$. 

Using the Poincar\'e inequality, one can prove that $u \in W^{1, p}(\Omega)$ then $u \in BMO(\Omega)$,  this  yields the exponential integrability of $u$ in \eqref{John-Nirenberg}. In fact,  for $u \in W^{1, p}(\Omega)$ with $p \ge d$, one can improve \eqref{John-Nirenberg}. First,  Morrey's inequality (see, e.g., \cite{BrezisF}) states that  $u \in C^{\alpha}(\Omega)$ with $\alpha = 1 - d/ p$ if $u \in W_0^{1, p}(\Omega)$ for $p> d$. Second,  Moser-Trudinger's inequality \cite{Mos, Tru,Pet,Poh} confirms that 
\begin{equation*}
\sup_{\| u\|_{W_0^{1, d}(\Omega) \le 1}} \int_{\Omega} e^{\alpha |u|^{d/ (d-1)}} \le C, 
\end{equation*}
for some positive constants $\alpha$ and $C$ depending only on $\Omega$. 

The goal of this paper is to understand  whether or not a better integrability  property of $u$ than \eqref{John-Nirenberg} inequality holds when $u \in L^p(\Omega)$ and $I_{\delta, p}(u, \Omega) < + \infty$.  It is worth noting that, for all $\delta > 0$, there exists $u \in L^\infty(\Omega) \setminus C(\bar \Omega)$ such that $I_{\delta, p}(u, \Omega) = 0$ for all $p \ge 1$. A simple example is the function $
u = \delta \mathds{1}_{B} \mbox{ in } \Omega, $ for some ball $B \Subset  \Omega$, where $\mathds{1}_O$ denotes the characteristic function of a subset $O$ of $\mR^d$. One can also show that there exists a function $u$ such that $I_{\delta}(u, \Omega) < + \infty$ and $u \not \in L^\infty(\Omega)$. An example for this is the function $u (x) = \bct{\ln \lambda}^{-1}\ln \ln |x|^{-1}$ for $x \in B_{1/e}$ and $\lambda > p/d$ (the verification is given in Section~\ref{sect-verification}).

In this work, we address the gap between the exponential integrability  \eqref{John-Nirenberg} and the boundedness  for functions $u$ with   $I_{\delta,p}(u, \Omega) < + \infty$ for some $\delta > 0$ and $p \ge d$. Our first result is 

 \begin{thm}\label{nonloc_supcr} \label{thm1}
Let $p> d \ge 1$, $\delta > 0$,  and let  $B$ be a an open ball of $\mR^d$. We have, 

$i)$ for $M> 0$ and $\alpha>0$, there exists a constant $0 \le \beta = \beta (\alpha, M) \le 1$ depending only on $M$ and  $\alpha$  such that 
\begin{align}\label{thm1-i}
\sup_{ |B|^{\frac{p-d}{d}} \delta^{-p} I_{\delta,p}(u,B) \leq M}  \fint_B e^{\alpha \bct{\frac{p}{d}}^{\beta \delta^{-1} |u-u_B|}} dx \le C, 
\end{align}

$ii)$ given $\alpha>0$, there exists a positive constant $M_0$ (small) depending only on $\alpha$, $d$, and $p$ such that 
\begin{align}\label{thm1-ii}
\sup_{ |B|^{\frac{p-d}{d}} \delta^{-p} I_{\delta,p}(u,B) \leq M_0}  \fint_B e^{\alpha \bct{ \frac{p}{d}}^{ \delta^{-1} |u-u_B|}} dx \le C. 
\end{align}
Here $C$  denotes a positive constant depending only on $d$, $p$, and $\alpha$.  

\end{thm}

As a consequence of Theorem~\ref{thm1}, we obtain 

\begin{pro}\label{pro1}
Let $p> d \ge 1$, $\delta > 0$,  and let  $\Omega$ be a smooth bounded open subset of $\mR^d$. We have, 

$i)$ for $M> 0$ and $\alpha > 0$,  there exists a constant $0 \le \beta = \beta (\alpha, M) \le 1$ depending only on $\alpha$ and $M$ such that 
\begin{align*}
\sup_{\delta^{-p} I_{\delta,p}(u,\Omega) \leq M} \int_\Omega e^{\alpha \bct{\frac{p}{d}}^{\beta \delta^{-1}|u|}} dx \le C_\Omega e^{\alpha \bct{\frac{p}{d}}^{\beta \delta^{-1} \|u \|_{L^1(\Omega)}}},  
\end{align*}

$ii)$ given $\alpha>0$, there exists a positive constant $M_0$ (small) depending only on $\alpha$, $d$, $p$, and $\Omega$ such that 
\begin{align*}
\sup_{\delta^{-p}I_{\delta,p}(u,\Omega) \leq M_0} \int_\Omega e^{\alpha \bct{ \frac{p}{d}}^{\delta^{-1}|u|}} dx \le C_\Omega e^{\alpha \bct{ \frac{p}{d}}^{\delta^{-1}\| u\|_{L^1(\Omega)}}}.  
\end{align*}
Here $C_\Omega$ denotes a positive constant depending only on $d$, $p$, $\alpha$,  and $\Omega$. 

\end{pro}

Here is a variant of $ii)$ of Theorem~\ref{thm1}.

\begin{thm} \label{thm2}Let $p = d \ge 1$, $\delta > 0$,  and let  $B$ be a an open ball of $\mR^d$. Given $\alpha>0$, there exists a positive constant $M_0$ (small) depending only on $\alpha$ and  $d$ such that 
\begin{align}\label{thm2-state-1}
\sup_{  \delta^{-d} I_{\delta,p}(u,B) \leq M_0} \fint_B e^{\alpha \delta^{-1}|u-u_B|} dx \le C, 
\end{align}
for some positive constant $C$ depending only on $d$ and $\alpha$.  
\end{thm}

\begin{rem}  Inequality \eqref{thm2-state-1} shares some similarities with John-Nirenberg's inequality but is different. In fact, 
fixing $\delta >0$,  as a consequence of \eqref{pro-Poincare}, we have 
$$
\| u \|_{BMO(B)} \le C (M +  \delta), 
$$
if $I_{\delta, p}(u) \le M$. Hence $\| u \|_{BMO(B)}$ does not generally converge to $0$ and \eqref{thm2-state-1} cannot be derived from \eqref{John-Nirenberg}.    

\end{rem}

As a consequence of Theorem~\ref{thm2}, we have 

\begin{pro}\label{pro2}
Let $p = d \ge 1$, $\delta > 0$,  and let  $\Omega$ be a smooth bounded open subset of $\mR^d$. Given $\alpha>0$, there exists a positive constant $M_0$ (small) depending only on $\alpha$, $d$, and $\Omega$  such that 
\begin{align*}
\sup_{\delta^{-p}I_{\delta,p}(u,\Omega) \leq M_0} \int_\Omega e^{\alpha \delta^{-1}|u|} dx \le C_\Omega e^{\alpha \delta^{-1} \| u\|_{L^1(\Omega)}},  
\end{align*}
for some positive constant $C_\Omega$ depending only on $d$, $\alpha$,  and $\Omega$.  
\end{pro}

The exponential growths in  \eqref{thm1-ii}  and \eqref{thm2-state-1} are optimal. In fact, we have 
\begin{pro} \label{pro3} Let $p \ge d \ge 1$, $\gamma > p/d$, and $\alpha > 0$,  and let  $B$ be a an open ball of $\mR^d$.  

i) If $p> d$ then for any $M>0$ there exists $u\in L^1(B)$ such that
\begin{equation}\label{pro3-1}
 I_{\delta, p} (u, B) \leq M  \quad \mbox{ and } \quad 
\int_B e^{\alpha \gamma^{\delta^{-1}  |u-{u}_B|}} dx  = + \infty.
\end{equation}

i) If $p= d$ then there exists a bounded sequence $(u_n) \subset L^1(B)$ such that
\begin{equation}\label{pro3-2}
\lim_{n \to + \infty} I_{\delta, p} (u_n, B) = 0 \quad \mbox{ and } \quad \lim_{n \to + \infty}
\int_B e^{\alpha \big (\delta^{-1}  |u_n-{u_n}_B| \big)^{\gamma}} dx  = + \infty. 
\end{equation}

\end{pro}


\section{Proofs of Theorems~\ref{thm1} and \ref{thm2}}
This section contains the proof of the Theorems \ref{thm1} and \ref{thm2}. We first establish two lemmas used in the proof of \eqref{thm1-i}, \eqref{thm1-ii}, and \eqref{thm2-state-1} and  then establish Theorems~\ref{thm1} and \ref{thm2}.

\subsection{Two useful lemmas} 

For $x \in \mR^d$ and $\rho > 0$, let $B_\rho(x)$ denote the ball in $\mR^d$ centered at $x$ and of radius $\rho$. We have 

\begin{lem}\label{lem1}
Let $d \ge 1$, $\lambda > 0$, and let  $E\subset F \subset \ro^d$ be measurable subsets of $\mR^d$ with $0<\md{E} <  \md{F} <\infty$. Let $\rho > 0$ be such that  $|E| = |B_\rho|$  and let $x \in \mR^d$
be such that $B_{2\rho}(x) \subset  F$. Then
\begin{equation}\label{lem1-state1}
\int_{F\setminus E} \frac{dy}{\md{x-y}^{\lambda}} \geq C_{d, \lambda} \md{E}^{1-\frac{\lambda}{d}}, 
\end{equation}
for some positive constant $C_{d, \lambda}$ depending only on $d$ and $\lambda$.  As a consequence, if  $p \ge 1$,  $|E| = |B_\rho|$ for some $\rho > 0$ and, for $p \ge 1$,  
$D$ is measurable subset of $F$ such that for almost every $x\in D$, the ball $B_{2\rho}(x) \subset  F$, then 
\begin{equation}\label{lem1-state2}
\int_{D}\int_{F\setminus E} \frac{dy \, dx }{\md{x-y}^{d+p}} \geq C_{d, p} |D| \md{E}^{-\frac{p}{d}}, 
\end{equation}
for some positive constant $C_{d, p}$ depending only on  $d$ and $p$.
\end{lem}

\begin{proof}  For $y \in \mR^d$,  we have
\begin{align*}
B_\rho (y) = \big( B_\rho(y) \setminus E \big)  \cup \big( B_\rho (y) \cap E \big) \text{ and } E= \big(E\setminus B_\rho(y) \big) \cup \big(E\cap B_\rho(y) \big).   
\end{align*}
Since  $|E|= \md{B_\rho(y)}$,  it follows that 
\begin{equation}\label{lem1-p1}
\md{E\setminus B_\rho (y)} =  \md{B_\rho(y)\setminus E} \mbox{ for } y \in \mR^d. 
\end{equation}
Fix $x$ such that $B_{2\rho}(x) \subset F$. We have 
\begin{align}\label{lem1-p2}
\int_{F \setminus E} \frac{dy}{\md{y - x}^{\lambda}} &=  \int_{(F\setminus E)\cap B_\rho(x)} \frac{dy}{\md{y - x}^{\lambda}} + \int_{(F\setminus E) \cap \big( F\setminus B_\rho(x) \big)}\frac{dy}{\md{y - x}^{\lambda}} \notag\\ &\geq  \frac{1}{\rho^{\lambda}} \md{\bct{F\setminus E} \cap B_\rho(x)} + \int_{(F\setminus E) \cap \big(F\setminus B_\rho(x) \big)}\frac{dy}{\md{y - x}^{\lambda}}. 
\end{align}
Then 
\begin{equation}\label{lem1-p3}
\md{\bct{F\setminus E} \cap B_\rho(x)}  \mathop{=}^{B_\rho(x) \subset F} \md{B_\rho(x) \setminus E} \mathop{=}^{\eqref{lem1-p1}} \md{E\setminus B_\rho(x)} \mathop{=}^{E \subset F} \md{\bct{F\setminus B_\rho(x)} \cap E}.
\end{equation}
Combining \eqref{lem1-p2} and \eqref{lem1-p3} yields 
\begin{equation*}
\int_{F\setminus E} \frac{dy}{\md{y - x}^{\lambda}} \geq \frac{1}{\rho^{\lambda}}\md{\bct{F\setminus B_\rho(x)} \cap E} + \int_{(F\setminus E) \cap \big(F\setminus B_\rho(x) \big)}\frac{dy}{\md{y - x}^{\lambda}}.
\end{equation*}
This yields 
\begin{align*}
\int_{F\setminus E} \frac{dy}{\md{y - x}^{\lambda}}  &\geq \int_{\big(F\setminus B_\rho (x) \big) \cap E} \frac{dy}{\md{y - x}^{\lambda}}+\int_{(F\setminus E) \cap \big( F\setminus B_\rho(x) \big)}\frac{dy}{\md{y - x}^{\lambda}}\\ 
&\mathop{\geq}^{E \subset F} \int_{F\setminus B_\rho}\frac{dy}{\md{y - x}^{\lambda}} \mathop{\geq}^{B_{2 \rho}(x) \subset F} \int_{B_{2\rho} \setminus B_\rho} \frac{dy}{\md{y - x}^{\lambda}} \ge C_{d, p} |E|^{1-\lambda/d}, 
\end{align*}
which is \eqref{lem1-state1}. 

Integrating \eqref{lem1-state1} w.r.t. $x$ in $D$, we obtain \eqref{lem1-state2}. 
\end{proof}

\begin{rem} 
A similar version of inequality \eqref{lem1-state1} has played crucial roles in deriving fractional versions of Sobolev \cite{DPV} and Hardy \cite{AdiAm}  inequalities.
\end{rem}

The following simple lemma is also used in the proof of Theorem~\ref{thm1}. 

\begin{lem}\label{lem2}
Let $d \ge 1$, $p>1$, $\delta > 0$, and let $O$ be a ball in $\mR^d$. Let $g \in L^1_{loc}(O)$. We have,  $k \in \mN_+$, 
\begin{align}\label{lrg value cntrl}
\mathop{\mathop{\int_O \int_O}}_{|u(x) - u(y)|> 2^k\delta} \frac{ \delta^p}{|x - y|^{d + p}} \leq 2^{-k(p-1)} \mathop{\mathop{\int_O \int_O}}_{|u(x) - u(y)|> \delta} \frac{ \delta^p}{|x - y|^{d + p}}.
\end{align}
\end{lem}

\begin{proof} By considering the function $u/\delta$ and  by the recurrence, it suffices to consider the case $\delta = 1$ and $k = 1$. We have 
\begin{multline*}
 \mathop{\iint_{O \times O}}_{|u(x) - u(y)| > 2} \frac{dx \, dy }{|x- y|^{d+p}} =  \mathop{\iint_{O \times O}}_{|u(x) - u(x/2+y/2)  + u(x/2+y/2)   - u(y)| > 1}  \frac{dx \, dy }{|x- y|^{d+p}}  \\[6pt] \le \mathop{\iint_{O \times O}}_{|u(x) - u(x/2+y/2) | > 1}  \frac{dx \, dy }{|x- y|^{d+p}}  +  \mathop{\iint_{O \times O}}_{|u(x/2+y/2) - u(y) | > 1}  \frac{dx \, dy }{|x- y|^{d+p}}.
\end{multline*} 
By a change of variables $z = x/2 + y/2$, we obtain 
 \begin{equation*}
 \mathop{\iint_{O \times O}}_{|u(x) - u(y)| > 2} \frac{dx \, dy }{|x- y|^{d+p}} 
 \le 2^{-(p-1)}\mathop{\iint_{O \times O}}_{|u(x) - u(y)| > 1} \frac{dx \, dy }{|x- y|^{d+p}},  
\end{equation*} 
which yields the conclusion for $\delta = 1$ and $k=1$. 
\end{proof}

\subsection{Proof of part i) of Theorem \ref{nonloc_supcr}.} In this proof, for notational ease, we denote $I_{\delta, p}$ by $I_\delta$ for $\delta > 0$.  Without loss of generality we can assume $B=B_1$, $u_B=0$, and $\delta = 1$. Define $\tilde{u}$ in $B_{3/2}$ by 
\begin{align*}
\T{u}(x) =\left\{ \begin{array}{cl}
\dsp u(x) &  \text{ if } x\in B_1 \\[6pt]
\dsp  u\bct{ \frac{(2-|x|) x}{|x|}}  & \text{ if } x\in B_{3/2} \setminus B_1.
\end{array} \right.
\end{align*}
We have, for all $\tau > 0$,  
\begin{align}\label{eqv}
 \Big| \Big\{x \in B_{3/2}; |\tilde u| \ge \tau \Big\} \Big| \le C \Big| \Big\{x \in B_1; |u| \ge \tau \Big\} \Big| 
\end{align}
and, see e.g.,    \cite[Lemma 17]{BreHmn}, 
\begin{align}\label{extn_ineq}
I_{1}\bct{\T{u}, B_{3/2}} \leq C I_{1} \bct{u,B_1}. 
\end{align}
Using John-Nirenberg's inequality, we have
\begin{equation}\label{smness_sup_lvl}
\Big|\Big\{x \in B_{3/2}; |\tilde u| \ge \ell  \Big\}  \Big| \le 1/8^d, 
\end{equation}
if $\ell \ge c_1 M$ for some $c_1 > 0$. 

We claim that, for  $\ell \ge c_1 M$   and $\lambda > 2$,  
\begin{align}\label{claim1}  
\Big|\Big\{x\in B_1: \md{u(x)}\geq \lambda \ell  \Big\} \Big| \leq c_2 I_{\ell} \bct{u,B_1} \Big|\Big\{x\in B_1: \md{u(x)} \geq (\lambda-1) \ell \Big\} \Big|^\frac{p}{d}.
\end{align}
In fact, fix an arbitrary $x \in B_{5/4}$ 
and let $\rho$ be such that $|B_\rho(x)|  = \big|\big\{x \in B_{3/2}; |u|  \ge (\lambda  -1 )\ell \big\} \big|$. Since $\lambda > 2$,  it follows from  \eqref{smness_sup_lvl} that $\rho < 1/ 8$, which yields $B_{2\rho}(x) \subset B_{3/2}$.  Applying Lemma~\ref{lem1} with $
D = \{x \in B_{5/4}; |\tu| \ge \lambda \ell \} \cap O,$ 
 $E =  \{x \in B_{3/2}; |\tu| \ge (\lambda- 1) \ell \} $ and $F  = B_{3/2}$, and using \eqref{eqv} and \eqref{extn_ineq},  we obtain \eqref{claim1}.  

Applying Lemma~\ref{lem2}, we have, for $k \in \mN$,  
\begin{equation*}
I_{2^k} (u, B_1) \le 2^{-k(p-1)} I_{1} (u, B_1) \le 2^{-k(p-1)} M. 
\end{equation*}
Fix $k_0$ be such that for $k \ge k_0$, one has $ 
c_2 2^{-k(p-1)} M \le e^{-2 \alpha}$, 
which yields 
\begin{equation}\label{ccc}
c_2 I_{2^k} (u, B_1) \le e^{- \alpha (p/d)^3}. 
\end{equation}

Set 
\begin{equation}\label{def-ell-0}
\ell_0 = \max \Big\{c_1 M, \Big(c_3 M e^{2\alpha} \Big)^{1/ (p-1)} \Big\}. 
\end{equation}
Then, for some $c_3$ larger than $c_2$,  
\begin{equation*}
\ell_0 \ge \max\{c_1 M, 2^{k_0} \}. 
\end{equation*}
Using \eqref{claim1}, \eqref{ccc}, and a standard iterative process, we have, for $\lambda \in \mN$ and  
\begin{equation}\label{key-part-thm1}
\Big| \Big\{x \in B_1; |u| > \lambda \ell_0 \Big\} \Big| \le e^{- \alpha (p/d)^{\lambda+2}} \Big| \Big\{x \in B_1; |u| > \ell_0 \Big\} \Big|.
\end{equation}
This implies 
\begin{equation*}
\int_{B_1} e^{\alpha(p/d)^{|u|}} \, dx \le \mathop{\int_{B_1}}_{|u| \ge \ell_0} e^{\alpha(p/d)^{|u|}} \, dx + \mathop{\int_{B_1}}_{|u| \le \ell_0} e^{\alpha(p/d)^{\ell_0}} \, dx \le C. 
\end{equation*}
This implies the conclusion of part $i)$ with $
\beta(\alpha, M) = \ell_0^{-1}$ where $\ell_0$ is given by \eqref{def-ell-0}. \qed

\subsection{Proof of part $ii)$ of Theorem~\ref{thm1}.}  The proof of part $ii)$ is in the spirit of part $i)$. In fact, noting that if $M_0$ is small enough then \eqref{key-part-thm1} holds with $\ell_0 = 1$. The conclusion then follows. \qed

\subsection{Proof of \eqref{thm2-state-1} of Theorem~\ref{thm2}.} The proof is similar to the one of part $ii)$ of Theorem~\ref{thm1} and is omitted. \qed

\subsection{Proof of Propositions~\ref{pro1} and \ref{pro2}} Propositions~\ref{pro1} and \ref{pro2} can be derived from Theorems~\ref{thm1} and \ref{thm2} respectively after using local charts and appropriately extending $u$ in a neighborhood of $\Omega$ (see, e.g., \cite[Lemma 17]{BreHmn}. The details are omitted. \qed

\section{Proof of  Proposition~\ref{pro3}} \label{sect-verification} 

Without loss of generality, one might assume that $B = B_{1/e}$ and $\delta  = 1$.  

\medskip 
\noindent {\bf Proof of  assertion \eqref{pro3-1}.}  Fix $\gamma > \lambda > p/d > 1$, set, for $x \in B_{1/e}$,  
\begin{equation*}
u(x) = g (|x|) \quad \mbox{ where } \quad   g(r)  =  (\ln \lambda)^{-1} \ln \ln (1/r)  \mbox{ for }  r \in I : = (0, 1/e). 
\end{equation*}
It is clear that $g \in L^1(I)$.  Using polar coordinates,  we have 
\begin{align}\label{part3-1}
I_{1,p}(u, B_{1/e}) &= \mathop{\int_0^{1/e} \int_0^{1/e} \int_{\mS^{d-1}} \int_{\mS^{d-1}}}_{|g(r_1)  - g(r_2)| > 1} \frac{r_1^{d-1} r_2^{d-1}}{|r_1 \sigma_1 - r_2 \sigma_2|^{p+d}} \, d \sigma_1 \, d \sigma_2 \, d r_1 \, d r_2.  \notag \\ &\leq  C_d \mathop{\int_0^{1/e} \int_0^{1/e} }_{|g(r_1)  - g(r_2)| > 1} \frac{r_1^{d-1} r_2^{d-1}}{|r_1  - r_2 |^{p+d}} \, d r_1 \, d r_2.
\end{align}

We have, for $0< r_1 <  r_2 < e^{-1}$,  
$$
|g(r_1) - g(r_2)| > 1 \mbox{ if and only if }  r_2 > r_1^{1/ \lambda} \mbox{ and } 0<r_1<e^{-\lambda},  
$$
this yields 
$$
\frac{r_1 r_2}{   (r_2 - r_1)^{1 + p/d}}  \le \frac{C r_1}{  r_2^{p/d}} \le C, 
$$  
for some positive constant $C$ depending only on $d$, $p$, and $\lambda$.  It follows that, for $0< r_1 <  r_2 < e^{-1}$ and $|g(r_1) - g(r_2)| > 1 $, 
\begin{align}\label{part3-1.2}
\frac{r_1^{d-1} r_2^{d-1}}{|r_1  - r_2 |^{p+d}} &= \bct{\frac{r_1 r_2}{\bct{r_2-r_1}^{\frac{p}{d}+1}}}^{d-1} \frac{1}{\md{r_1-r_2}^{\frac{p}{d}+1}}    \leq \frac{C}{\md{r_1-r_2}^{\frac{p}{d}+1}}.
\end{align}
We derive from \eqref{part3-1} and \eqref{part3-1.2} that 
\begin{equation}\label{part3-2}
I_{1,p}(u, B_{1/e})  \le C I_{1,\frac{p}{d}}(g, I). 
\end{equation}

We have
\begin{multline}\label{part3-3}
I_{1,p/d}(g, I) = 2 \mathop{\mathop{\iint_{I \times I}}_{|g(r_1)-g(r_2)|>1}}_{r_1 < r_2} \frac{1}{|r_2 - r_1|^{1 + \frac{p}{d}}} \, dr_1 \, d r_2 \\[6pt]  \le C \int_0^{e^{- \lambda}} \left(  \frac{1}{\big(r_1^{1/ \lambda} - r_1\big)^\frac{p}{d}} - \frac{1}{\big(e^{-1} - r_1 \big)^\frac{p}{d}}\right) \, dr_1 < + \infty,   
\end{multline}
since $r_1^{1/ \lambda} - r_1 \ge C r_{1}^{1/ \lambda}$ and $e^{-1} - r_1 \ge C$ for $r_1 \in (0, e^{-\lambda})$. 

On the other hand, for any $\tau \in I $, we have, with $\rho = \frac{\ln \gamma}{\ln \lambda}-1$,
\begin{equation}\label{part3-4}
\int_{B_{\tau}}  e^{\alpha \gamma^{ g}} \, dx  = \int_0^\tau e^{\alpha \big(\log r^{-1} \big)^{ (1+\rho)}} r^{d-1} \, dr = + \infty, 
\end{equation}
since $\lim_{r  \to 0_+} \big(\log r^{-1} \big)^{1+\rho}/  \log r^{-1} = + \infty$. 

Set, for $0<  \tau < e^{-1}$,  
$$
u_{\tau} (x) = u (\tau x ) \mbox{ for } x \in B_{e^{-1}}. 
$$
Then 
\begin{equation}\label{part3-5}
I_{1,p}(u_{\tau}, B_{e^{-1}}) = \tau^{p-d}  I_{1, p} (u, B_{\tau e^{-1}}) \to 0 \mbox{ as } \tau \to 0. 
\end{equation}

Combining \eqref{part3-4} and \eqref{part3-5} yields the conclusion since for any $M>0$ we can choose $\tau>0$ small enough so that $I_{1,p}(u_{\tau}, B_{e^{-1}})\leq M$. \qed

\medskip 

\noindent{\bf Proof of  assertion \eqref{pro3-2}.}  Let $n \in \mN$ large and fix $1< q < \gamma$ and denote $q' = q/ (q-1)$.  Define
\begin{equation*}\label{moser func}
u_{n}(x) = g_n(|x|) \mbox{ where } g_n (r)= \left\{\begin{array}{cl} \dsp \ln^\frac{1}{q} n  & \dsp \text{ if } 0\leq r\leq \frac{1}{n},\\[6pt] 
\dsp \frac{ \ln \bct{\frac{1}{r}}}{ \ln^\frac{1}{q'}n} &  \dsp \text{ if } \frac{1}{n} \leq r \leq \frac{1}{e}.
\end{array}\right. 
\end{equation*}
As in \eqref{part3-2}, we have
\begin{equation}\label{thm2-ppp1}
I_{1,d}(u_n, B_{1/e})  \le C I_{1,1}(g_n, I), 
\end{equation}
where $I = (0, 1/e)$.  

We now estimate $I_{1,1}(g_n, I)$.  Denote  $\I_n = (0, 1/n)$, and $\J_n = I \setminus \I_n$.  We have 
\begin{equation}\label{cal of nonloc qtty}
I_{1,1}\bct{g_{n},I} = 2 {\mathcal I}_1 +  {\mathcal I}_2. 
\end{equation}
where 
\begin{equation*}
\mcal{I}_1=  \mathop{\iint_{I_n \times J_n}}_{|g_n(r_1) - g_n(r_2)|  > 1} \frac{1}{|r_1-r_2|^{2 }} \, dr_1 \, dr_2 \quad \mbox{ and } \quad 
\mcal{I}_2 =  \mathop{\iint_{J_n \times J_n}}_{|g_n(r_1) - g_n(r_2)|  > 1} \frac{1}{|r_1-r_2|^{2}} \, dr_1 \, dr_2.
\end{equation*}

We next estimate $\mathcal{I}_1$ and $\mathcal{I}_2$. We begin with $\mathcal{I}_1$.  For $(r_1, r_2) \in \I_n \times \J_n$, we have $|g_n(r_1) - g_n(r_2)| > 1$ if and only if 
\begin{align*}
\md{\ln^\frac{1}{q} n -\frac{\ln \bct{\frac{1}{r_2}}}{\ln^\frac{1}{q'} n }} > 1, \text{ this implies } \frac{1}{e}\geq r_2>a_n : =  \frac{e^{\bct{\log n}^\frac{1}{q'}}}{n}.
\end{align*}
It follows that 
\begin{equation}\label{est of I}
\mcal{I}_1 =  \int_0^\frac{1}{n} \int_{a_n}^\frac{1}{e} \frac{dr_2 dr_1}{|r_1-r_2|^2}  \le  \ln \left( \frac{a_n}{a_n - 1/n} \right)  \to 0 \mbox{ as } n \to + \infty. 
\end{equation}
We next deal  with $\mathcal{I}_2$.  For $(r_1, r_2) \in \J_n \times \J_n$ with $r_2 \ge r_1$, we have $|g_n(r_1) - g_n(r_2)| > 1$ if and only if 
\begin{align*}
\dsp  \md{\frac{\ln \bct{\frac{1}{r_1}}-\ln \bct{\frac{1}{r_2}}}{\ln^\frac{1}{q'} n}} >1, \mbox{ this implies } \frac{1}{e} \geq r_2 > r_1 b_n  \text{ and } \frac{1}{n} \leq r_1 < \frac{1}{eb_n} \mbox{ with } b_n =e^{\ln^\frac{1}{q'}n}. 
\end{align*}
We then have 
\begin{align}\label{est of II}
\mcal{I}_2 &=  2 \int_{1/n}^{1/ (e b_n)} \int_{r_1 b_n}^\frac{1}{e} \frac{dr_2}{(r_2-r_1)^2} dr_1 \le   2 \int_{1/n}^{1/(e b_n)}  \frac{1}{r_1 (b_n -1)} \, dr_1  \to 0 \mbox{ as } n \to + \infty. 
\end{align}
Combining  \eqref{cal of nonloc qtty}, \eqref{est of I},  and \eqref{est of II}  yields 
\begin{align*}
\lim_{n\rightarrow \infty} I_{1,1}\bct{g_{n}, I} =0.
\end{align*}
which yields, by \eqref{thm2-ppp1}, 
\begin{align}\label{thm2-ppp2}
\lim_{n\rightarrow \infty} I_{1,d}\bct{u_{n},B_{1/e}} =0.
\end{align}

We have 
\begin{equation*} 
 \int_\frac{1}{n}^\frac{1}{2} \frac{\ln \bct{\frac{1}{r}}}{ \ln^\frac{1}{q'} n} dr =  \frac{1}{\bct{\log n}^\frac{1}{q'}}\bct{\frac{1}{2}\ln (2e) -\frac{1}{n}\ln (ne)}  \to 0  \text{ as } n\rightarrow \infty.
\end{equation*}
This implies 
\begin{equation}\label{thm2-ppp3}
0 \le \lim_{n \to + \infty} \int_{B_{1/e}} u_n \, dx \le \lim_{n \to + \infty} \int_{I} g_n \, dx =0. 
\end{equation}
On the other hand, since $\gamma>q$,  we have 
\begin{equation}\label{thm2-ppp4}
\int_{I} r^{d-1} e^{\alpha g_n^r}  \geq \int_0^\frac{1}{n} r^{d-1} e^{\alpha \ln^{\gamma/q} n } = \frac{1}{d n^{d}} e^{\alpha \ln^{\gamma/q} n }  \to + \infty \mbox{ as } n \to + \infty.  
\end{equation}
The conclusion now follows from \eqref{thm2-ppp2}, \eqref{thm2-ppp3}, and \eqref{thm2-ppp4}.
\qed


\footnotesize
\begin{thebibliography}{99}

\bibitem{AdiAm} Adimurthi and A~Mallick, \emph{A Hardy type inequality on fractional order Sobolev spaces on the Heisenberg group}, Ann. Sc. Norm. Super. Pisa Cl. Sci. (5)  {\bf 18} (2018), 917--949.


\bibitem{ADM} L. Ambrosio, G. De Philippis, and L. Martinazzi, \emph{Gamma-convergence of nonlocal perimeter functionals}, Manuscripta Math. {\bf 134} (2011), 377--403.

\bibitem{AGMP1}  C. Antonucci, M. Gobbino, M. Migliorini,  and N. Picenni, \emph{On the shape factor of interaction laws for a non-local approximation of the Sobolev norm and the total variation}. C. R. Math. Acad. Sci. Paris {\bf 356} (2018), 859--864.
 
\bibitem{AGMP2}   C. Antonucci, M. Gobbino, M. Migliorini, N. Picenni, Optimal constants for a non-local approximation of Sobolev norms and total variation, arXiv: 1708.01231.

\bibitem{AGP} C. Antonucci, M. Gobbino, N. Picenni, On the gap between gamma-limit and pointwise limit for a non-local approximation of the total variation, arXiv:1712.04413. 
 

\bibitem{AK} G.~Aubert and P. Kornprobst, \emph{Can the nonlocal characterization of Sobolev spaces by Bourgain et al. be useful for solving variational problems?}, SIAM J. Numer. Anal. {\bf 47} (2009), 844--860. 




\bibitem{BoNg1} J. Bourgain and H.-M.  Nguyen, 
\textit{A new characterization of Sobolev spaces}, 
C. R. Math. Acad. Sci. Paris  {\bf 343} 1 (2006).

\bibitem{B1} H.~Brezis, {\sl How to recognize constant functions. Connections with Sobolev spaces},
Volume in honor of M. Vishik, Uspekhi Mat. Nauk {\bf 57} (2002), 59--74;
English translation in Russian Math. Surveys {\bf 57} (2002),
693--708.

\bibitem{BBH} F. Bethuel, H. Brezis, F. Heilein, Gnzburg-Landau vortices. Reprint of the 1994 edition [MR1269538]. Modern Birkhäuser Classics. Birkhäuser/Springer, Cham, 2017. xxix+158 pp.

\bibitem{BBM1} J.~Bourgain, H.~Brezis and P.~Mironescu, 
\textit{Another look at Sobolev spaces}, 
 Optimal control and partial differential equations,  IOS, Amsterdam (2001), 439--455.


\bibitem{BBM2} J.~Bourgain, H.~Brezis and P.~Mironescu, \emph{{Limiting embedding theorems for $W^{s,p}$ when $s \uparrow 1$ and applications}}, J. Anal. Math.
\textbf{87} (2002), 77--101.

\bibitem{BBM3} J.~Bourgain, H.~Brezis and P.~Mironescu, \emph{{Lifting, degree, and distributional Jacobian revisited}},  Comm. Pure Appl. Math.
\textbf{58} (2005), 529--551.



\bibitem{BBNg} J.~Bourgain, H.~Brezis, H.-M. Nguyen, A new estimate for the topological degree. C. R. Math. Acad. Sci. Paris {\bf 340}  (2005),  787--791. 


\bibitem{BrezisNguyen}
H.~Brezis and H.-M. Nguyen, \emph{{On a new class of functions related to VMO}}, C. R. Acad. Sci. Paris
\textbf{349} (2011), 157--160.

\bibitem{BrezisF} H. Brezis, Functional analysis, Sobolev spaces and partial differential equations. Universitext. Springer, New York, 2011. xiv+599 pp.


\bibitem{BN-Two}
H.~Brezis and H.-M. Nguyen,  \emph {Two subtle convex nonlocal approximation of the $BV$-norm}, Nonlinear Anal.,  {\bf 137} (2016), 222--245. 

\bibitem{BN-BBM}
H.~Brezis and H.-M. Nguyen,  \emph {The BBM formula revisited}, Atti Accad. Naz. Lincei Rend. Lincei Mat. Appl., {\bf 27} (2016), 515--533. 








\bibitem{BreHmn} H.~Brezis and H.-M. Nguyen, 
\textit{Non-local functionals related to the total variation and connections with image processing}, 
Ann. PDE {\bf 1} 4 (2018).


\bibitem{BuadesCollMorel2} A.~Buades, B.~Coll and J. M. Morel,  \emph{{Neighborhood filters and PDE's}},  Numer. Math. \textbf{105} (2006),  1--34. 



\bibitem{CV}
 L. Caffarelli and E. Valdinoci, \emph{Uniform estimates and limiting arguments for nonlocal minimal surfaces},  Calc. Var. PDE  {\bf 41}  (2011), 203--240.




\bibitem{Davila} J. Davila, \emph{On an open question about functions of bounded variation}, Calc. Var. Partial Differential Equations {\bf 15} (2002), 519--527. 

\bibitem{DPV} E. Di Nezza, G. Palatucci, and E. Valdinoci, 
\textit{Hitchhiker's guide to the fractional Sobolev spaces}, 
Bull. Sci. Math. {\bf 136} 5 (2012), 521--573.



\bibitem{FS} C. Fefferman and E. Stein, \emph{$H^p$ spaces of several variables},  Acta. Math. {\bf 129}  (1972) 137--193. 




\bibitem{JN} F. John and L. Nirenberg, 
\emph{On functions of bounded mean oscillation},  
Comm. Pure Appl. Math. {\bf 14} (1961) 415--426. 



\bibitem{LeoniSpector}
G.~Leoni and D.~Spector, \emph{{Characterization of Sobolev and BV Spaces}}, J. Funct.  Anal.  \textbf{261} (2011),
  2926--2958,  \emph{Corrigendum to ``Characterization of Sobolev and BV Spaces"}, J. Funct. Anal. {\bf 266} (2014), 1106--1114.



\bibitem{Ma} V. G. Mazya, Sobolev Spaces. Springer, Berlin (1985). 

\bibitem{Mos} J. Moser,
 \textit{ A sharp form of an inequality by N. Trudinger}, 
 Indiana Univ. Math. J. {\bf20} 11 (1971), 1077--1092.


\bibitem{NgSob1}
H.-M. ~Nguyen, \emph{{Some new characterizations of Sobolev spaces}}, J. Funct.
  Anal. \textbf{237} (2006), 689--720.

\bibitem{NgDegree1} H.-M.~Nguyen, \emph{Optimal constant in a new estimate for the degree},  J. Anal. Math. {\bf 101} (2007), 367--395. 


\bibitem{NgGammaCRAS}
H.-M. ~Nguyen, \emph{{$\Gamma$-convergence and Sobolev norms}}, C. R. Acad. Sci. Paris  \textbf{345} (2007), 679--684.


\bibitem{NgSob2}
H.-M. ~Nguyen, \emph{{Further characterizations of Sobolev spaces}}, J. European
  Math. Soc. \textbf{10} (2008), 191--229.


\bibitem{NgGamma}
H.-M. ~Nguyen, \emph{{$\Gamma$-convergence, Sobolev norms, and BV functions}}, Duke
  Math. J. \textbf{157} (2011), 495--533.

\bibitem{NgSob3}
H.-M. ~Nguyen, \emph{{Some inequalities related to Sobolev norms}}, Calc. Var.
  Partial Differential Equations \textbf{41} (2011), 483--509.


\bibitem{NgDegree2} H.-M.~Nguyen,  \emph{A refined estimate for the topological degree},  C. R. Math. Acad. Sci. Paris {\bf 355} (2017),  1046--1049. 

\bibitem{Ng-Squ1} H.-M. Nguyen,  A. Pinamonti, M. Squassina, and E. Vecchi, \emph{New characterizations of magnetic Sobolev spaces},   Adv. Nonlinear Anal.   {\bf 7} (2018), 227--245. 

\bibitem{Ng-Squ2} H.-M. Nguyen and M. Squassina, {Marco }\emph{Logarithmic Sobolev inequality revisited},  C. R. Math. Acad. Sci. Paris {\bf 355} (2017),  447--451.


\bibitem{Ng-Squ3} H.-M. Nguyen; M. Squassina, \emph{On Hardy and Caffarelli-Kohn-Nirenberg inequalities}, Journal d'Analyse Mathematique. 2017, to appear. 

\bibitem{Ng-Squ4} H.-M. Nguyen; M. Squassina, \emph{On anisotropic Sobolev spaces},  Commun. Contemp. Math. {\bf 21} (2019),  1850017, 13 pp.

\bibitem{Squ1} M. Squassina, B. Volzone, \emph{Bourgain-Brezis-Mironescu formula for magnetic operators}, C. R. Math. Acad. Sci.Paris {\bf 354} (2016), 825--831.



\bibitem{Pet} J. Peetre,  
\textit{Espaces d'interpolation et theoreme de Soboleff, Ann. Inst.}
Fourier (Grenoble) {\bf16} (1966), 279--317.


\bibitem{Poh} S. I. Pohozhaev, 
\emph{The Sobolev imbedding in the case pl = n.} 
Proc. Tech. Sci. Conf. on Adv. Sci. Research 1964-1965, Mathematics Section, 158--
170, Moskov. Energet. Inst., Moscow 1965

\bibitem{Po1}
A.~Ponce, \emph{{A new approach to Sobolev spaces and connections to $\Gamma$-convergence}}, Calc. Var.
  Partial Differential Equations \textbf{19} (2004), 229--255.

 \bibitem{Po2} A.~Ponce, \emph{An estimate in the spirit of Poincar\'e's inequality},  J. Eur. Math. Soc. {\bf 6} (2004),  1--15. 
 
 \bibitem{PoSp} A.~Ponce and S.~ Daniel \emph{On formulae decoupling the total variation of BV functions.},  Nonlinear Anal.  {\bf 154} (2017),  241--257. 
 
 
  \bibitem{Sch} J. V.  Schaftingen, \emph{Estimates by gap potentials of free homotopy decompositions of critical Sobolev maps},   arXiv:1811.01706v2 . 

 

\bibitem{Tru} N. S. Trudinger, 
\textit{On imbeddings into Orlicz spaces and some applications.} 
J. Math. Mech. {\bf17} (1967), 473--483.


\bibitem{Yaroslavsky1}
  L.~P.~Yaroslavsky, \emph{{Digital Picture Processing. An introduction}}, Springer, 1985.


\bibitem{YaroslavskyEden} L.~P.~Yaroslavsky and M.Eden, \emph{{Fundamentals of Digital Optics}}, Springer, 1996.

\end{thebibliography}
\end{document}